\documentclass[12pt]{amsart}
\usepackage{amssymb}
\usepackage{epsfig}
\usepackage{amscd}

 \newcommand{\RR}{\mathbb{R}}

\newcommand{\XX}{X}
\newcommand{\Aut}{{\text{Aut}}}

\newcommand{\br}{{\mathbb R}}

\newcommand{\RP}{{\mathbb{{RP}}}}
\newcommand{\ra}{{\rightarrow}}
\newcommand{\rand}{\partial X}

\newcommand{\xo}{o}
\newcommand{\aL}{\mathfrak{a}}
\newcommand{\bL}{\mathfrak{b}}

\newcommand{\kL}{\mathfrak{k}}
\newcommand{\gL}{\mathfrak{g}}
\newcommand{\nL}{\mathfrak{n}}

\newcommand{\pL}{\mathfrak{p}}

\newcommand{\diag}{\mbox{\rm Diag}}

\newcommand{\Tr}{\mbox{\rm Tr}}

\newcommand{\be}{\begin{eqnarray*}}
\newcommand{\ee}{\end{eqnarray*}}

\newcommand{\st}{\mbox{such}\ \mbox{that}\ }

      \theoremstyle{plain}
      \newtheorem{thm}{Theorem}[section]
      \newtheorem{lemma}[thm]{Lemma}
      \newtheorem{co}[thm]{Corollary}
      \newtheorem{Prop}[thm]{Proposition}

      \theoremstyle{definition}
      \newtheorem{Def}[thm]{Definition}
      \newtheorem*{example}{Example}

\let \cal \mathcal
\begin{document}
\title[Structure of limit set]
{On the limit set of Anosov representations}
\date{}

\author{Inkang Kim}
\address{School of Mathematics,
KIAS, Hoegiro 85, Dongdaemun-gu Seoul, 130-722, Republic of Korea}
\email{inkang@kias.re.kr}
\author{Sungwoon Kim}
\email{sungwoon@kias.re.kr}

\footnotetext[3]{I. Kim gratefully acknowledges the partial support
of KOSEF Grant (R01-2008-000-10052-0).}
\begin{abstract}
We study the limit set of discrete subgroups arising from Anosov representations. Specially we study the limit set of discrete groups arising from strictly convex real projective structures and Anosov representations from a finitely generated word hyperbolic group
into a semisimple Lie group.
\end{abstract}

\maketitle
\tableofcontents


\section{Introduction}

There has been intensive study about the limit set of rank one
symmetric spaces. Nonetheless it is still mysterious how the limit
set of higher rank symmetric spaces looks like.
In \cite{Ha}, it is analyzed that some Tits neighborhoods of
parabolic fixed points of nonuniform lattices in higher rank
semisimple Lie groups do not include conical limit points, which
is a sharp contrast to real rank one case. In this paper we try to
describe some examples and results related to the linear action
and a geometric structure which arise as a discrete subgroup of
$\mathrm{SL}(d,\br)$. This example naturally arises as convex
projective structures on surfaces. More generally such  groups
appear in so-called Hitchin component of representation variety of
surface group in $\mathrm{SL}(d,\RR)$.

A strictly convex real projective structure on surfaces is a
generalization of hyperbolic structure. Nonetheless if we look at
the action on $\mathrm{SL}(3,\RR)/\mathrm{SO}(3)$ instead of on
$\RP^2$, it is not obvious that we can get the same phenomena as
in $\mathrm{SL}(2,\RR)/\mathrm{SO}(2)$. Yet it shares many
parallel properties since the Hilbert metric associated to the
projective structure is more or less hyperbolic like. This is our
motivation to study limit sets in
$\mathrm{SL}(3,\RR)/\mathrm{SO}(3)$ arising from such a geometric
structure and attempt to classify the limit points. Another
motivation is to compare the action on
$\mathrm{SL}(3,\RR)/\mathrm{SO}(3)$ and the natural linear action
on $\RR^3$. The latter relation will be  investigated in a future paper.

We begin by defining types of limit points on the geometric
boundary of general symmetric spaces. The notions of radial limit point and
horospherical limit point are introduced by Albuquerque in \cite{Al}
and Hattori in \cite{Ha} as in the theory of Kleinian groups.

\begin{Def}[\cite{Ha}] Let $\Gamma$ be a discrete subgroup of
$\mathrm{Isom}(X)$ where $X$ is a symmetric space of
noncompact type. A limit point $\xi \in \partial X$ is \emph{horospherical} if
there exists a sequence $\gamma_n\in \Gamma$ so that for any
horoball $B$ based at $\xi$, $\gamma_n \xo$ is contained  in $B$
for all large $n$.\end{Def}

\begin{Def}[\cite{Al}]
A limit point $\xi\in\rand$ is called a \emph{radial limit point} if
there exists a sequence $\gamma_n\in \Gamma$ \st
$\gamma_n\xo$ converges to $\xi$ in the cone topology and remains
at a bounded distance of the union of closed Weyl chambers with
apex $\xo$ containing the geodesic ray $\sigma_{\xo,\xi}$.
\end{Def}

The notion of conical limit point  is also defined in \cite{Ha}.
The condition is stronger than being a radial limit point.
It is easily seen that every limit point for a uniform lattice is radial.
Hattori \cite{Ha} characterizes exactly radial (conical) limit points for $\mathbb{Q}$-rank $1$ lattices.
He also shows that every limit point for a finitely generated generalized Schottky group in
$\mathrm{SL}(2,\mathbb{R})\times \mathrm{SL}(2,\mathbb{R})$ is horospherical.
It seems to be difficult to classify limit points of general discrete subgroup of higher rank symmetric space.
In this paper, we prove that

\begin{thm}\label{thm:1.3}
Let $\Gamma$ be a discrete subgroup of $\mathrm{SL}(3,\mathbb R)$
arising from a convex real projective structure on a closed surface. Let $X$ be the
symmetric space associated to $\mathrm{SL}(3,\mathbb R)$.
Then the limit set $\Lambda_\Gamma$ of $\Gamma$ in the Furstenberg
boundary of $X$ is homeomorphic
to $S^{1}$ and the limit set $L_\Gamma$ in the geometric boundary
of $X$ splits as a product $S^{1}\times I$ where $I$ is the closed
interval identified with the directions of the limit cone.
\end{thm}

In addition to Theorem \ref{thm:1.3}, we characterize radial limit
points of $\Gamma$ in the geometric boundary of
$\mathrm{SL}(3,\br)/\mathrm{SO}(3)$.

\begin{thm}\label{thm:1.4}
Let $\Gamma$ be a discrete subgroup of $\mathrm{SL}(3,\mathbb R)$
arising from a convex real projective structure on a closed surface.
Every limit point of $\Gamma$ is horospherical. Furthermore there
is only one radial limit point in each Weyl chamber at infinity
with nonempty limit set of $\Gamma$.
\end{thm}

It is well known that the Hitchin component of the representation
variety of a surface group in $\mathrm{SL}(3,\mathbb R)$ is equal to
the deformation space of convex projective structures on the
surface \cite{CG}. Due to Theorem \ref{thm:1.3} and \ref{thm:1.4},
one can see how the structure of limit set is changed in the
Hitchin component for $\mathrm{SL}(3,\mathbb R)$ as follows: Let
$\Gamma_0$ be a discrete subgroup of $\mathrm{SL}(3,\mathbb R)$
arising from a hyperbolic structure on a closed surface.
Then it is a standard fact that the limit set of
$\Gamma_0$ in the geometric boundary of $\mathrm{SL}(3,\br)/\mathrm{SO}(3)$ is homeomorphic to a circle $S^1$ and the limit set in
any Weyl chamber at infinity, if nonempty, consists of a point.
Moreover, it can be easily seen that every limit point of $\Gamma_0$ is a radial limit point.

When $\Gamma_0$ is deformed to a discrete subgroup of $\mathrm{SL}(3,\mathbb R)$ arising from a convex real projective structure on the surface,
the limit set in each Weyl chamber at infinity with nonempty limit
set is changed from a point to an interval and thus limit set is
changed from a circle to a cylinder. Even though the limit set in
each Weyl chamber at infinity suddenly increases from a point to
an interval, it turns out due to Theorem \ref{thm:1.4} that the
set of radial limit points in each Weyl chamber at infinity does
not increase. Indeed, there exists only one point in each interval which is a
radial limit point and hence, the number of radial limit points in
each Weyl chamber at infinity is preserved under the deformation
of $\Gamma_0$. To our knowledge, this is the first example of a
concrete description of limit set in higher rank symmetric space,
except for limit sets of lattices.

All the machinery to show the above theorems work equally well for any Anosov representations, see section \ref{Anosov}.
Hence we have
\begin{thm}\label{thm:1.5}
Let
$\rho:\Gamma\ra G$ be a Zariski dense discrete $P$-Anosov representation from a word hyperbolic group $\Gamma$ where $P$ is a minimal parabolic subgroup of a semisimple Lie group $G$.  Then the geometric limit set is isomorphic to the set $\partial \Gamma\times \partial \cal L_{\rho(\Gamma)}$. Furthermore in each Weyl chamber intersecting the geometric limit set nontrivially, there is only one radial limit point.
\end{thm}
Here $\partial \cal L_{\rho(\Gamma)}$ is the set of directions of limit cone $ \cal L_{\rho(\Gamma)}$. See section \ref{limitcone} for definitions.

\section{Preliminaries}

Let $G$ be a higher rank semisimple real Lie group and $X$ an
associated  symmetric space. For each $\xi \in \partial X$, there is
an associated parabolic group $P_\xi$ which is a stabilizer of $\xi$ in
$G$. Then $P_\xi$ has a generalized Iwasawa decomposition
$P_\xi=N_\xi A_\xi K_\xi$, where $K_\xi$ is a subgroup of an isotropy group of a
fixed point $\xo$ in $X$, $A_\xi\xo$ is the union of parallels to a
geodesic $l$ connecting $\xo$ and $\xi$, and $N_\xi$ is the
horospherical subgroup determined only by $\xi$. If $\xi$ is a regular
point, $P_\xi$ becomes a minimal parabolic subgroup. In this case the
group $G$ has a Iwasawa decomposition $G=KAN$, where $K$ is an
isotropy group of $\xo$, $A\xo$ is a maximal flat, and $N$ is a
nilpotent group stabilizing the regular point $\xi$. A choice of a
Weyl chamber $\aL^+$ in $\aL$, the Lie algebra of $A$,
 determines a positive root and accordingly a fundamental
system $\Upsilon$ of roots. $A^+\xi$ is also called a Weyl chamber
with an apex $\xi$ where $A^+=\exp({\aL^+})$. A choice of a subset
$\Theta \subset \Upsilon$ determines a face of $\aL^+$
$$ \aL^{\Theta}=\{H \in {\overline {\aL^+}} \ | \ \alpha(H)=0, \ \alpha \in
\Theta\}.$$
So any singular element $\xi \in \partial X$ can be represented by an element
in $\aL^{\Theta}$ for some $\Theta \subset \Upsilon$.

If $\xi$ is a singular point, $K_\xi$ is
a centraliser of  $\aL^{\Theta}$ in $K$,  $A \subset A_\xi$, and
$N_\xi \subset N$. In this case
$$G=KA_\xi N_\xi$$ is called a generalised Iwasawa decomposition. See
\cite{Eb} for details.

In terms of Lie algebras, we can describe parabolic subgroups as follows.
Let $\gL=\kL\oplus \pL$ be the Cartan decomposition, $\aL\subset \pL$ a maximal abelian subset as before.
The adjoint action of $\aL$ gives rise to a decomposition of $\gL$ into eigenspaces
$$ \gL = \bigoplus_{\alpha\in \Sigma} \gL_\alpha,\ \text{where}\ \gL_\alpha=\{x\in\gL: [a,x]=\alpha(a)x \text{ for } \forall a\in\aL\}.$$ Here $\Sigma$ is the system of restricted roots of $\gL$. Let $N_K(\aL)$ and $Z_K(\aL)$ be the normalizer and the centralizer of $\aL$ in $K$. The Weyl group $W=N_K(\aL)/Z_K(\aL)$ acts on $\aL$ and on $\Sigma$. A unique element $\omega_{op}\in W$ sending $\Sigma^-$ to $\Sigma^+$ induces an involution
$$\iota:\Sigma^+\ra \Sigma^+,\ \alpha\mapsto -\omega_{op}(\alpha),$$
called the opposite involution $\iota(\Upsilon)=\Upsilon$.

The subalgebra
$$\nL^+=\bigoplus_{\alpha\in\Sigma^+}\gL_\alpha$$ is nilpotent and $N=\exp(\nL^+)$ is unipotent.
The subgroup $B=Z_K(\aL)AN$ is a minimal parabolic subgroup with its Lie algebra $\bL^+=\gL_0\oplus \nL^+$. Similarly one can define $N^-,B^-$ using negative roots. The group $B^-$ is conjugate to $B^+$. In general, parabolic subgroups of $G$ are conjugate to subgroups containing $B^+$. A pair of parabolic subgroups is opposite if their intersection is  a reductive group. The conjugacy classes of parabolic subgroups are in one to one correspondence with subsets $\Theta \subset \Upsilon$. For each $\Theta$, let $\aL_\Theta=\cap_{\alpha\in\Theta} ker \alpha$ and $M_\Theta=Z_K(\aL_\Theta)$ its centralizer in $K$. Then
$$P^+_\Theta=M_\Theta A N \ \text{and}\ P^-_\Theta=M_\Theta A N^-$$ are opposite parabolic subgroups. Any pair of opposite parabolic subgroup is conjugate to $(P^+_\Theta,P^-_\Theta)$ for some $\Theta\subset \Upsilon$. The intersection $L_\Theta=P^+_\Theta\cap P^-_\Theta$ is the common Levi component of $P^+_\Theta$ and $P^-_\Theta$. The group $M_\Theta$ is a maximal compact subgroup of $L_\Theta$.

Note that $P^-_\Theta$ is conjugate to $P^+_{\iota(\Theta)}$. In particular $P^+_\Theta$ is conjugate to its opposite if and only if $\Theta=\iota(\Theta)$. In our case, we will deal with minimal parabolic subgroups $B^+,B^-$, hence they are opposite and conjugate.

A \emph{geometric boundary} (or ideal boundary) $\partial{X}$ of
$X$ is defined as the set of equivalence classes of geodesic rays
under the equivalence relation that two rays are equivalent if
they are within finite Hausdorff distance of each other. For any
point $x\in X$ and any ideal point $\xi \in
\partial{X}$, there exists a unique unit speed ray starting from $x$ which
represents $\xi$. The pointed Hausdorff topology on rays emanating
from $x\in X$ induces a topology on $\partial{X}$. This topology
does not depend on the base point $x$ and is called the \emph{cone
topology} on $\partial{X}$.

\begin{example}\label{keyexample} Let $G=\mathrm{SL}(d,\RR)$ and $\gL$ its Lie
algebra, the set of traceless $(d,d)$-matrices. The inner product
$\langle Y, Z \rangle= \Tr (YZ^t)$ is a positive definite inner
product on $\gL$ which is a usual inner product on $\RR^{d^2}$.
The associated symmetric space $X$ can be identified with the set
of positive definite symmetric matrices with determinant 1, and
$\mathrm{SL}(d,\RR)$ acts on it by conjugation $x \ra gxg^t$. The
isotropy group of the identity matrix $I\in \mathrm{SL}(d,\RR)$ is
$\mathrm{SO}(d)$, hence $X=\mathrm{SL}(d,\RR)/\mathrm{SO}(d)$. We
will denote $\xo$ the class of $I$ in
$\XX=\mathrm{SL}(d,\RR)/\mathrm{SO}(d)$. If $\kL$ denotes the Lie
algebra of $\mathrm{SO}(d)$, then $\gL=\kL\oplus \pL$ is a Cartan
decomposition where $\pL$ is identified with $T_\xo X$.
Furthermore,
$$\aL=\left\{\diag(\lambda_1, \lambda_2,  \ldots , \lambda_d) \ \Big| \ \sum_{i=1}^d
\lambda_i=0 \right\}$$
is a maximal abelian subspace of $\pL$ and we choose
$$\aL^+=\left\{\diag(\lambda_1, \lambda_2,  \ldots , \lambda_d) \ \Big| \ \sum_{i=1}^d
\lambda_i=0, \ \lambda_1 > \lambda_2 > \cdots > \lambda_d \right\}.$$
\end{example}

For $Y\in\pL$ with $\Vert Y\Vert=1$ we denote by $\sigma_Y$ the unique
(unit speed) geodesic such that $\sigma_Y(0)=\xo$ and
$\sigma'_Y(0)=Y$. In particular
$\sigma_Y(s)=\exp(Ys)\cdot \xo$.  Any point $\xi \in \partial X$ is
realised as $\sigma_Y(\infty)$ for some $Y\in \pL$ with $\Vert Y\Vert=1$.
Let $\lambda_1(Y) > \cdots > \lambda_k(Y)$ be distinct eigenvalues of $Y$
and $E_i(Y)$ be the eigenspace of $Y$ corresponding to $\lambda_i(Y)$.
Set $V_i(Y)=\oplus_{j=1}^i E_j(Y)$. Then we get a flag
$$V_1(Y) \subset \cdots \subset V_k(Y)=\RR^d.$$
If $m_i$ is the dimension of $E_i(Y)$, the following two conditions
\begin{eqnarray*}
\text{(i)}\hskip .5 in \sum_{i=1}^k m_i \lambda_i(Y)&=&0\\
\text{(ii)}\hskip .4 in \sum_{i=1}^k m_i \lambda_i(Y)^2&=&1
\end{eqnarray*}
 are satisfied due to the fact that $Y$ is traceless and a unit
vector.

In this way one gets a one-to-one correspondence between $\partial
X$ and the set of flags with two conditions. If $F(Y)$ is a flag
associated with a point in $\partial X$,  the action of $g \in
\mathrm{SL}(d,\RR)$ is just $gF(Y)$. The typical example is when
$Y\in \aL$ is a diagonal matrix with distinct entries. The
corresponding flag to $\mathfrak{a}^+$ is just
$$\langle e_1 \rangle\subset \langle e_1, e_2 \rangle \subset\cdots \subset \langle
e_1,\cdots,e_{d-1}\rangle \subset \RR^d.$$ Changing eigenvalues
corresponds to moving around in the same Weyl chamber. The adjacent
Weyl chamber with $\lambda_2>\lambda_1>\lambda_3>\cdots>\lambda_d$
is the flag
$$\langle e_2 \rangle\subset
\langle e_1, e_2 \rangle \subset \cdots \subset \langle
e_1,\cdots,e_{d-1}\rangle \subset \RR^d$$ and the opposite Weyl
chamber with $\lambda_d>\lambda_{d-1}>\cdots>\lambda_1$ is
\begin{eqnarray*}
\langle e_d \rangle\subset \langle e_d, e_{d-1}\rangle \subset
\cdots\subset \RR^d.
\end{eqnarray*}
Note that two Weyl chambers corresponding to two flags $V_1 \subset \cdots \subset V_d$ and
$W_1 \subset \cdots \subset W_d$ are opposite if for any $i+j=d$,
\begin{eqnarray}\label{opposite}
V_i \oplus W_j= \mathbb{R}^d
\end{eqnarray}
We refer the reader to \cite[Section 2.13]{Eb} for more details about this.

In our case, it is particularly interesting when $\xi$ is a
singular point. Let $H_1:=\sqrt{(d-1)/d}\;\diag(1,-1/(d-1),\ldots,
-1/(d-1))\in\overline{\aL_1^+}$ be a diagonal matrix with last
$d-1$ entries the same. This vector in $\pL$ denotes a (maximal)
singular direction by a geodesic $\sigma_{H_1} $ starting from
$\xo$ and ending at a point $\xi_1\in \partial X$ which we will
denote by $\infty$. Let $r$ be a singular geodesic ray which is
the image of $\sigma_{H_1}$ and
$\mathrm{SL}(d,\RR)=\mathrm{SO}(d)A_\infty N_\infty$ a generalised
Iwasawa decomposition. There is a nice description of the set
$A_\infty$ and $N_\infty$.
\begin{enumerate}
\item{$g \in N_\infty$ if and only if (i) $g_{ij}=0$ whenever
    $\lambda_j \geq \lambda_i$ and (ii) $g_{jj}=1$. So
    $N_\infty$ are upper triangular with 1's on the diagonal.}
\item{$g \in A_\infty$ if and only if $g_{ij}=0$ whenever
    $\lambda_i \neq \lambda_j$ and $g$ is symmetric positive
    definite.}
\end{enumerate}

It is not difficult to see that the union of parallels to $r$ is
$A_\infty \cdot \xo $ and is isometric to $\RR \times
\mathrm{SL}(d-1,\RR)/\mathrm{SO}(d-1)$, see \cite{Kim}. Note here
that the $\RR$-factor is exactly the singular geodesic $r$. The
orbit $N_\infty I$ is perpendicular to this set.  A level set of a
Busemann function centered at $\infty$ is $r(t_0) \times
\mathrm{SL}(d-1,\RR)/\mathrm{SO}(d-1)$ together with
$N_\infty\cdot (r(t_0)\times
\mathrm{SL}(d-1,\RR)/\mathrm{SO}(d-1))$. In matrix form an element
of $A_\infty$ looks like
\[ \left [ \begin{matrix}
       \mu &  0 \\
         0     &  M    \end{matrix} \right ] \] where $\mu>0$, and $M$ is
     a positive definite symmetric $(d-1, d-1)$-matrix with
         determinant equal to $1/\mu$. An
     element of $K_\infty$ also looks like
\[ \left [ \begin{matrix}
          \pm1  & 0 \\
           0 & M' \end{matrix} \right] \] where $M' \in \mathrm{O}(d-1)$.

Now, we recall the definition of the limit set of a discrete group.

\begin{Def}
Let $\Gamma$ be a discrete subgroup of $G$. The \emph{(geometric) limit set}
$L_\Gamma$ of $\Gamma$ is defined by $L_\Gamma= \overline{\Gamma \cdot \xo} \cap \partial X$.
\end{Def}

We remark that this definition does not depend on the chosen base point $\xo$ and
can be extended to isometry groups of arbitrary Hadamard manifolds.
We call each point of $L_\Gamma$ a \emph{limit point} of $\Gamma$.

\begin{Def}
Let $\sigma :[0,\infty) \rightarrow X$ be a geodesic ray.
The Busemann function $b(\sigma) : X \rightarrow \mathbb{R}$ associated with $\sigma$ is given by
$$b(\sigma) (x)=\lim_{t \rightarrow \infty} (d(x, \sigma(t))-t) \ \text{for } x \in X.$$
For any real number $C$, we call the set $b(\sigma)^{-1}((-\infty,C))$ an open \emph{horoball}
centered at $\sigma(\infty)$, and the set $b(\sigma)^{-1}(C)$ a \emph{horosphere} centered at $\sigma(\infty)$.
\end{Def}

It is not easy to classify limit points for arbitrary discrete groups of higher rank symmetric space but
we here give an example for which every limit point is a radial limit point as follows.

\begin{example} Let $Y=X_1\times X_2$ be a product of $\mathbb{R}$-rank one symmetric spaces. Let $\Gamma\subset \mathrm{Isom}^+(X_1)\times
\mathrm{Isom}^+(X_2)$ be a group acting freely on $Y$. Set $\Gamma_1\subset
\mathrm{Isom}^+(X_1)$ be the projection of $\Gamma$ to the first
factor. If the geometric limit set $L_\Gamma$ consists of only
regular points, then by \cite{DKL}, $\Gamma=\{(\gamma,\phi\gamma)
\ | \ \gamma\in \Gamma_1\}$ is a graph group for some type
preserving isomorphism $\phi$. Furthermore
$L_\Gamma=\Lambda_\Gamma\times [a,b]$ where $\Lambda_\Gamma$ is a
limit set in Furstenberg boundary and
$$[a,b]=\overline{\left\{\frac{l(\gamma)}{l(\phi\gamma)} \ \Big| \ \gamma\in\Gamma_1\
\text{is hyperbolic} \right\} }\subset \RR$$ is a closed interval. Here $l(\gamma)$
denotes the translation length of $\gamma$. Furthermore
for any $[(\xi_1,\xi_2),p]\in L_\Gamma$, there is a sequence of
hyperbolic isometries $\{(\gamma^1_i,\gamma^2_i)\}$ so that
$$(\gamma^1_i)^+\rightarrow \xi_1,\ (\gamma^2_i)^+\rightarrow \xi_2,\ \frac{l(\gamma^1_i)}{l(\gamma^2_i)}\rightarrow p$$
where $\gamma^+$ denotes the attractive fixed point of $\gamma$
i.e., $\gamma^+ = \lim_{j \rightarrow \infty} \gamma^j o$. This
implies that the Weyl chambers determined by $(\gamma^1_i)^+$ and
$(\gamma^2_i)^+$ converges to the Weyl chamber determined by
$\xi_1$ and $\xi_2$, and the slope of the invariant axis of
$(\gamma^1_i,\gamma^2_i)$ converge to $p$. Then it is not
difficult to show that $[(\xi_1,\xi_2),p]$ is a radial limit
point.
\end{example}

\section{Limit cone, Jordan decomposition and Cartan decomposition}\label{limitcone}
An element $g$ of  a real reductive connected linear group can be
uniquely written
$$g=ehu$$ where $e$ is elliptic (all its complex eigenvalues have
modulus 1), $h$ is hyperbolic (all the eigenvalues are real and
positive) and $u$ is unipotent ($u-I$ is nilpotent), and all three
commute \cite{Hel}. This decomposition is called the Jordan
decomposition of $g$. If $G=KAN$ is any Iwasawa decomposition of a
semisimple Lie group $G$, $e$ is conjugate to an element in $K$,
$h$ is conjugate to an element in $A$, and $u$ is conjugate to an
element in $N$ \cite{Hel}. The \emph{translation length}
$l(\alpha)$ of an isometry $\alpha$ is defined by
$$\inf_{x\in G/K} d(x,\alpha(x)).$$
It is shown in \cite{Eb} that when $G$ is a real semisimple Lie
group, for $g\in G$, if $g=ehu$ is the Jordan decomposition, then
$l(g)=l(h)$.


Fix a closed Weyl chamber $A^+\subset G$ and denote $\lambda:G\ra
A^+$ the natural projection induced from the Jordan
decomposition: for $g\in G$, $\lambda(g)$ is a unique element in
$A^+$ which is conjugate to the hyperbolic component  $h$ of
$g=ehu$. Note that $\lambda(g^n)=\lambda(g)^n$ since
$g^n=e^nh^nu^n$. For $g\in \mathrm{SL}(n,\RR)$, $\log \lambda(g)$
is the vector in ${\mathfrak a}^+$ whose coordinates are
logarithms of the absolute values of eigenvalues of $g$ arranged
in a decreasing order. Since $l(\lambda(g))=|\log(\lambda(g))|$,
we get the following corollary.
\begin{co}Let $G$ be a real semisimple Lie group, and  $g=ehu$ in $G$ in its
Jordan decomposition. Then
$$|\log(\lambda(g))|=l(\lambda(g))=l(h)=l(g).$$
\end{co}

Let $\omega_{op}$ be the element in the Weyl group of ${\mathfrak
a}$ which maps ${\mathfrak a}^+$ to $-{\mathfrak a}^+$. The
opposite involution $\iota:{\mathfrak a}^+ \ra {\mathfrak a}^+$ is
defined to be: for $X \in {\mathfrak a}^+$,
$\iota(X)=Ad\omega_{op}(-X)$. The \emph{limit cone} ${\cal L}_\Gamma$
of $\Gamma$ is the smallest closed cone in ${\mathfrak a}^+$
containing $\log(\lambda(\Gamma))$. Benoist \cite{Be} showed

\begin{thm} Let $G$ be a real linear connected semisimple Lie group.
If $\Gamma$ is a Zariski-dense sub-semigroup of $G$, then
the limit cone is convex and its interior is nonempty. If $\Gamma$
is a Zariski dense subgroup, then ${\cal L}_\Gamma$ is invariant
under the opposite involution $i$. Moreover the limit set of
$\Gamma$ in any Weyl chamber at infinity, if nonempty, is naturally
identified with the set of directions in ${\cal L}_\Gamma$.
\end{thm}

\section{Well-displacing representations and U-property}\label{well}
 Let $\gamma$ be an isometry of a metric space $Y$. We recall
that the translation length of $\gamma$ is $d_Y(\gamma)= \inf_{x\in
Y} d(x, \gamma(x))$. We observe that $d_Y(\gamma)$ is an invariant
of the conjugacy class of $\gamma$. If $C_\Gamma$ is the Cayley
graph of a group $\Gamma$ with set of generators $S$ and word length
$\|\ \|_S$, the displacement function is called the translation
length and is denoted by $\ell_S$
$$\ell_S(\gamma) = \inf_\eta\|\eta \gamma\eta^{-1}\|_S.$$
Note that this is equal to the number of generators involved to write
$\gamma$ in a cyclically reduced way. We finally say the action by
isometries on $X$ of a group $\Gamma$ is \emph{well-displacing},
if given a set $S$ of generators of $\Gamma$, there exist positive
constants $A$ and $B$ such that
$$d_Y(\gamma) \geq A\ell_S(\gamma) - B.$$ This definition does not
depends on the choice of $S$. From the definition, it is immediate
that for $\rho:\Gamma\rightarrow \mathrm{Isom}(Y)$ to be a well-displacing
representation, it must be discrete and faithful, and the image
consists of only hyperbolic isometries.

For hyperbolic groups, well-displacing
action is equivalent to that the orbit map is a quasi-isometric
embedding from the Cayley graph $C_\Gamma$ to $Y$ \cite{gu}, i.e,
for any $x\in Y$, there exist constants $A$ and $B$ so that
$$A^{-1}\|\gamma\|-B\leq d(x,\gamma(x)) \leq A\|\gamma\|+B.$$

We say that a finitely generated group has \emph{$U$-property} if there
exists finitely many elements $g_1,\ldots,g_p$ of $\Gamma$, positive
constants $A$ and $B$ such that for any $\gamma\in \Gamma$,
$$\|\gamma\|\leq A \sup_i \ell(g_i\gamma)+B.$$
It is shown \cite{gu} that a closed surface group has $U$-property.
If a representation of a group with $U$-property is well-displacing,
then
{\setlength\arraycolsep{2pt}
\begin{eqnarray*}
d(x,\gamma x) & \geq & \sup d(g_i^{-1}x, \gamma x)-\sup d(x, g_i^{-1}x) \\
& \geq & \alpha\sup \ell(g_i\gamma)-\beta-\sup d(x,g_ix) \\
& \geq & \alpha A\|\gamma\|-B\alpha-\beta-\sup d(x,g_ix).
\end{eqnarray*}}
Also if $\gamma=\gamma_1\cdots\gamma_k$ for $\gamma_i$ in generating
set,
$$d(x,\gamma x)\leq d(x,\gamma_k x)+d(\gamma_k x,
\gamma_{k-1}\gamma_k x)+\cdots + d(\gamma_2\cdots\gamma_k x, \gamma
x)\leq C \|\gamma\|$$ for some $C$, which shows that the orbit map
is a quasi-isometric embedding. Labourie \cite{La08} showed that Hitchin representations are well displacing.
Hence the orbit map of any Hitchin representation is a quasi-isometric embedding.

\section{Limit set of convex real projective
surfaces}\label{convex} In this section we give an example of a
limit set which is a topological circle in Furstenberg boundary.
The example comes from a strictly convex real projective structure
on a closed surface. The main source is from \cite{Kim}. As far as
we know, this is the first example of a concrete description of a
limit set in higher rank symmetric space, which is quite
interesting in its own right.

A real projective structure on a manifold $M$ is a maximal atlas
$\{U_i,\phi_i\}$ into $\RP^d$ so that the transition functions
$\phi_i\circ \phi_j^{-1}$ are restrictions of projective
automorphisms of $\RP^d$. A \emph{(strictly) convex real
projective} surface
 $S$ is $\Omega/\Gamma$ where $\Omega$ is a (strictly) convex
domain in $\RP^2$ and $\Gamma$ is a discrete subgroup of
$\Aut(\RP^2)$. Up to taking a subgroup of index two, we can assume
that $\Gamma \subset \mathrm{SL}(3,\br)$.

An element $g\in
\mathrm{GL}(d,\br)$ is called \emph{proximal} if
$\lambda_1(g)>\lambda_2(g)$ where $\lambda_1(g)\geq
\lambda_2(g)\geq \cdots \geq \lambda_d(g)$ is the sequence of
modules of eigenvalues of $g$ repeated with multiplicity. It is
called \emph{biproximal} if $g^{-1}$ is also proximal. A proximal
element is called \emph{positively proximal} if the eigenvalue
corresponding to $\lambda_1(g)$ is a positive real number. When
$S=\Omega/\Gamma$ is a closed convex real projective surface with
$\chi(S)<0$, Kuiper \cite{Kui} showed that $\Omega$ is strictly
convex with $\partial\Omega$ at least $C^1$, and every
homotopically nontrivial closed curve on $S$ is freely homotopic
to a unique closed geodesic (in the Hilbert metric) which
represents a positively biproximal element in
$\mathrm{SL}(3,\br)$.

From now on, we fix $S=\Omega/\Gamma$ a strictly convex real
projective closed surface such that any element in $\Gamma\subset
\mathrm{SL}(3,\br)$ is positively biproximal and we set $X=\mathrm{SL}(3,\mathbb R)/\mathrm{SO}(3)$.
We will show that
the limit set of $\Gamma$ in the Furstenberg boundary of $X$ is a circle.
Note that by Benoist \cite{Be} if $\Omega$ is not an ellipsoid (in
ellipsoid case $\Omega$ is a real hyperbolic 2-plane), $\Gamma$ is
Zariski dense in $\mathrm{SL}(3,\br)$ and the intersection of the
limit set with a Weyl chamber at infinity, if nonempty, has nonempty interior by
\cite{Be1}. So the limit set itself cannot be homeomorphic to a
circle. By this reason we consider limit set in the Furstenberg
boundary of $X$. For a general connected semisimple Lie group with trivial center and no compact factors, the \emph{Furstenberg boundary} is homeomorphic to
$G/P$ where $P$ is a minimal parabolic subgroup. \emph{A limit set $\Lambda_\Gamma$
of $\Gamma$ in the Furstenberg boundary} is the closure of the
attracting fixed points of elements in $\Gamma$. See \cite[Lemme
2.6]{Be}.

\begin{thm}\label{circle}Let $\Gamma$ be a discrete subgroup of $\mathrm{SL}(3,\mathbb R)$
arising from a convex real projective structure on a closed surface. Then the limit set $\Lambda_\Gamma$ of $\Gamma$
in the Furstenberg boundary of $\mathrm{SL}(3,\br)/\mathrm{SO}(3)$ is a circle.
\end{thm}
\begin{proof}
Let $S=\Omega/\Gamma$ be the convex real projective closed surface.
For any $g\in \Gamma$, $g$ has an attracting  fixed point
$\langle v^+ \rangle$ and a repelling fixed point $\langle v^-
\rangle$ in $\partial \Omega$ corresponding to $\lambda_1(g)$ and
$\lambda_3(g)$. Since $g$ is biproximal, all the eigenvalues are
positive reals and $\lambda_1(g)>\lambda_2(g)>\lambda_3(g)$. If
$v_0$ denotes the eigenvector corresponding to $\lambda_2(g)$,
$g$ fixes a flag
$$\langle v^+ \rangle \subset \langle v^+,v_0 \rangle \subset
\br^3$$ with eigenvalues $\lambda_1(g),\lambda_2(g),\lambda_3(g)$.
Also since $g$ leaves invariant $\langle v^+,v_0 \rangle$ and
$\langle v^-,v_0\rangle$, $\langle v_0 \rangle$ is a unique
intersection point of a line  $\langle v^+ ,v_0\rangle$ and a line
 $\langle v^- ,v_0\rangle$ in $\RP^2$. Since these lines
cannot pass through the interior of $\Omega$ (otherwise $\langle
v_0 \rangle$ is on $\partial\Omega$, then $g$ will have three
fixed points on $\partial\Omega$, which is not allowed), these
lines are tangent to $\partial\Omega$ at $\langle v^+ \rangle$
and $\langle v^- \rangle$ respectively, so $\langle v_0 \rangle$
is uniquely determined by $\langle v^+ \rangle$ and $\langle v^-
\rangle$. Note that in this flag, $\langle v^+,v_0 \rangle$ is a
line tangent to $\partial\Omega$ at $\langle v^+ \rangle$ in
$\RP^2$.

This eigenvalue-flag pair is a limit point of $g^{n>0} I$ in
$\partial X$, which is a regular point in the Weyl chamber of
$\partial X$ corresponding to $\langle v^+ \rangle \subset \langle
v^+,v_0 \rangle \subset \br^3$ since
$\lambda_1(g)>\lambda_2(g)>\lambda_3(g)$, and also this flag is an
attracting fixed point of $g$ in the Furstenberg boundary of $X$. So
for any $g\in \Gamma$, $g$ determines a unique fixed point
$\langle v^+ \rangle$  on $\partial\Omega$, and in turn this
determines a unique flag $\langle v^+ \rangle \subset \langle
v^+,v_0 \rangle \subset \br^3$ where $\langle v^+, v_0 \rangle$ is
a line in $\RP^2$ tangent to $\partial\Omega$ at $\langle v^+
\rangle$, which is an attracting fixed point of $g$ in the Furstenberg
boundary of $X$.

Note that in this correspondence, for any $g\in \Gamma$, the
attracting fixed point of $g$ in Furstenberg boundary is a flag
determined by $\langle v^+ \rangle \in \partial\Omega$ and a  line
through $\langle v^+ \rangle$ tangent to $\partial \Omega$. So a
point in the closure of attracting fixed points of elements in
$\Gamma$ in the Furstenberg boundary, is determined by a tangent line
through some point on $\partial\Omega$. But a point on
$\partial\Omega$ determines a unique tangent line since $\Omega$
is strictly convex  and $\partial\Omega$ is $C^1$ by Kuiper
\cite{Kui}. This shows that the limit set of $\Gamma$ in the
Furstenberg boundary, which is the closure of attracting fixed
points of elements in $\Gamma$ by \cite{Be}, is homeomorphic to
$\partial\Omega$ which is a circle. Note here that we can apply
Benoist's theorem since $\Gamma$ is Zariski dense either in
$\mathrm{SL}(3,\br)$ (if $\Omega$ is not an ellipse) or in
$\mathrm{SO}(2,1)$ (when $\Omega$ is an ellipse).
\end{proof}

As observed already, the hyperbolic plane sits inside $X$ as
$\RR\times \mathbb H^2$ where $\RR$ is a singular geodesic. Fix $x_0\in
\mathbb H^2$, then for any geodesic $l$ through $x_0$, $\RR\times l$ is a
maximal flat in $X$. Also $l(\infty)$ is a barycenter of a Weyl
chamber so that $Td(\RR(\infty),l(\infty))=\pi/2$. If $\Gamma$ is
a Fuchsian group, then the geometric limit set is just a circle. Hence one
can expect that if we perturb a Fuchsian group to a convex
projective structure, then the geometric limit set would be a
cylinder. Indeed, by the result of \cite[Section 7.5]{Be1}, one
can show that the limit set $L_\Gamma$ of $\Gamma$ can be
identified with $\Lambda_\Gamma \times \partial\cal L_\Gamma$ where $\cal L_\Gamma$
is a limit cone. See also \cite{DKL} and \cite{Ga}. So the limit
set $L_\Gamma$ in $\partial X$ is identified to a cylinder. More rigorously,

\begin{thm}\label{limitset}Let $S=\Omega/\Gamma$ be a compact strictly convex real
projective surface. Then the limit set $\Lambda_\Gamma$ of $\Gamma$
in the Furstenberg boundary is homeomorphic to $S^{1}$ and the limit
set $L_\Gamma$ in the geometric boundary splits as a product
$S^{1}\times I$ where $I$ is the closed interval. Furthermore for
every pair of distinct points $p, q\in\partial \Omega$, the
corresponding Weyl chambers $W_p,W_q$ are opposite.
\end{thm}
\begin{proof} By Benoist \cite{Be1}, for each $p\in\partial \Omega$,
$W_p\cap L_\Gamma$ can be identified with the set of directions of the limit cone $\cal L_\Gamma$. The only
obstruction for $L_\Gamma$ to be a cylinder is that $\cal L_\Gamma$
contains a singular direction and for two distinct
$p,q\in\partial\Omega$, $W_p,W_q$ are adjacent. For two distinct points $p,q\in
\partial\Omega$, choose a
sequence $\gamma_n\in \Gamma$, so that ${\gamma_n^+}$ and ${\gamma_n^-}$ converge to $p$ and $q$ respectively. Since two lines $\langle\gamma^+_n,
\gamma^0_n\rangle$ and $\langle \gamma^-_n,\gamma^0_n\rangle$
intersect at $\langle \gamma^0_n\rangle$, and $\partial\Omega$ is
$C^1$, $\gamma^0_n\ra v_0$ in $\RP^2$. Hence these lines converge
to $\langle p, v_0 \rangle,\ \langle q, v_0\rangle.$
Two flags $$\langle p \rangle \subset \langle p,v_0 \rangle \subset \mathbb R^3 \text{ and }
\langle q \rangle \subset \langle q,v_0 \rangle \subset \mathbb R^3$$
correspond to two Weyl chambers $W_p$ and $W_q$ respectively.
Since $p\neq q$,
two Weyl chambers $W_p$ and $W_q$ are opposite due to equation (\ref{opposite}).
The same argument holds for any distinct pairs $p,q \in \partial \Omega$.
This shows that Weyl chambers corresponding to two distinct pairs
are opposite, and consequently they are not adjacent. Hence the
geometric limit set $L_\Gamma$ is homeomorphic to
$\Lambda_\Gamma\times \partial\cal L_\Gamma$.
\end{proof}

Indeed, the first statement in Theorem \ref{limitset} easily
follows from the result of Sambarino \cite{Sam} that the limit
cone of any discrete group in the Hitchin component of
$\mathrm{SL}(d,\mathbb R)$ is contained in the interior of the
Weyl chamber. This implies that every limit point in $L_\Gamma$ is regular.

\section{Characterisation of limit points in convex real
projective surfaces}

Let $K$ be an infinite compact metrisable topological space. Suppose that a
group $G$ acts by homeomorphism on $K$. A group $G$ is
said to be a \emph{convergence group} if the induced action on the
space of distinct triples of $K$ is properly discontinuous, or equivalently if  a given sequence of
distinct $g_i \in G$, there are points $c$ and $b$ of $K$ and a subsequence $(g_{n_i})$ such that
$$g_{n_i}(z) \rightarrow b$$ uniformly outside neighborhoods of $c$.
Convergence groups acting on the standard sphere or ball of
$\mathbb{R}^n$ were first introduced by Gehring and Martin
\cite{GeMa}. Then Freden \cite{Fe} and Tukia \cite{Tukia94}
generalized the notion of convergence group to groups acting on
spaces other than the sphere or the ball and having the
convergence property. For instance, a group of isometries of a
Gromov hyperbolic space can be extended to the Gromov boundary as
a convergence group \cite{Tukia94}. For further discussion, see
\cite{Bo}, \cite{Fe}, \cite{Tukia98}.

The limit set $L_G$ of $G$ is the set of limit points, where a
limit point is an accumulation point of some $G$-orbit in $K$. The
limit set is the unique minimal closed nonempty $G$-invariant
subset of $K$ and $G$ acts properly discontinuously on the
$K\backslash L_G$. A point $z\in L_G$ is said to be a
\emph{conical limit point} if there is a sequence $(g_n)$ of
distinct elements of $G$ such that, for every $x \in L_G
\backslash \{z\} $, the sequence $(g_n x, g_n z)$ is relatively
compact in $L_G \times L_G \backslash \Delta_G$ where
$\Delta_G=\{(y,y) \ | \ y\in L_G \}$.

Let $\Gamma$ be a closed surface group and $C_\Gamma$ be the Caley
graph of $\Gamma$. Since $\Gamma$ is a hyperbolic group, the
Gromov boundary $\partial C_\Gamma$ of $\Gamma$ is well defined up to H\"older homeomorphism.
Hence the group $\Gamma$ is a convergence group acting by
homeomorphism on the Gromov boundary. Furthermore, it is well
known that $L_\Gamma=\partial C_\Gamma$ and every point of
$\partial C_\Gamma$ is a conical limit point \cite{Tukia94}.

Let $S=\Omega / \Gamma$ be a strictly convex real projective
closed surface for $\Gamma \subset \mathrm{SL}(3,\br)$. Since
$\Gamma$ is a closed surface group and acts cocompactly on
$\Omega$, there is a canonical identification of $\partial
C_\Gamma$ with $\partial \Omega$. Hence, it is obvious that every
point of $\partial \Omega$ is a conical limit point with respect
to the action of $\Gamma$ on $\partial \Omega$.

Recall that the Furstenberg boundary
$\partial_F X$ of $X$ can be identified with the set of
equivalence classes of Weyl chambers in maximal flats in $X$.
Here, two Weyl chambers $W_1, W_2$ are called \emph{equivalent} if $$d_H(W_1,
W_2)< +\infty,$$ where $d_H$ is the Hausdorff distance on subsets
of $X$. For each chamber $W$, denote its equivalence class by $[W]$.
The usual angle at a point $x$ in $X$ subtended by the vectors of the centers of gravity of
Weyl chambers in the unit tangent space $S_xX$ gives rise to a metric in $\partial_F X$.
For more details, see \cite[Section 3.8]{Eb}.

Let $\langle v\rangle$ be a point in $\partial \Omega \subset \RP^2$. As we
observed in the proof of Theorem \ref{circle}, $\langle v\rangle$ uniquely determines the
flag $\langle v\rangle \subset \langle v,v_0\rangle \subset \br^3$  where $\langle v,v_0\rangle$ is the
$2$-dimensional plane corresponding to the unique tangent line to
$\partial \Omega$ at $\langle v\rangle$. This flag determines a Weyl chamber,
denoted by $W_{\langle v\rangle}$. Now define a map $\phi : \partial \Omega
\rightarrow
\partial_F X$ by $$\phi(\langle v\rangle)= [W_{\langle v\rangle}]$$ for $\langle v\rangle \in \Omega$.
It can be easily seen that this map is a $\Gamma$-equivariant homeomorphism onto its
image $\Lambda_\Gamma=\phi(\partial \Omega)$. Due to this
$\Gamma$-equivariant homeomorphism $\phi :
\partial \Omega \rightarrow \Lambda_\Gamma$, every point of $\Lambda_\Gamma$ is a conical limit point
with respect to the action of $\Gamma$ on $\Lambda_\Gamma$.

\begin{Prop}\label{prop:6.1}
Let $\Gamma$ be a discrete subgroup of $\mathrm{SL}(3,\mathbb R)$
arising from a convex projective structure on a closed surface.
Every Weyl chamber at infinity with nonempty limit set of $\Gamma$
in the geometric boundary of $\mathrm{SL}(3,\br)/\mathrm{SO}(3)$
contains at least one radial limit point.
\end{Prop}

\begin{proof}
Due to the homeomorphism $\phi : \partial \Omega \rightarrow
\Lambda_\Gamma$, every point in $\Lambda_\Gamma$ is of the form
$[W_p]$ for some $p \in \partial \Omega$. It is sufficient to
prove that the Weyl chamber $W_p$ at infinity contains a radial
limit point.

As observed before, every point in
$\Lambda_\Gamma$ is a conical limit point. Hence there is a
sequence $(\gamma_n)$ of distinct elements of $\Gamma$ such that
for every $[W_q] \in \Lambda_\Gamma \backslash \{[W_p]\}$, the
sequence $(\gamma_n [W_p], \gamma_n [W_q])$ is relatively compact
in $\Lambda_\Gamma \times \Lambda_\Gamma \backslash \Delta_\Gamma$
where $\Delta_\Gamma$ is the diagonal in $\Lambda_\Gamma \times
\Lambda_\Gamma$. Thus, by passing to a subsequence, we can assume that $\gamma_n [W_p]$
converges to $[W_a]$ and $\gamma_n [W_q]$ converges to $[W_b]$ for
some distinct points $a, b \in \partial \Omega$.

For a Weyl chamber $W$, write $W(\infty)=\overline{W} \cap
\partial X$ where $\overline{W}$ is the closure in the
compactification $X\cup \partial X$ of $X$. Note that if two
chambers $W_1$ and $W_2$ are equivalent,
$W_1(\infty)=W_2(\infty)$. Thus
$[W](\infty)=W(\infty)$ is well defined for each equivalence class $[W]$. It is a
standard fact that if two Weyl chambers in the Furstenberg
boundary are opposite, there exists a unique maximal flat
connecting them. Since any two distinct Weyl chambers in
$\Lambda_\Gamma$ are opposite by Theorem \ref{limitset}, there is
a unique maximal flat connecting $\gamma_n [W_p](\infty)$ and
$\gamma_n [W_q](\infty)$. Let $F_n$ denote such maximal flat.
Denote by $F_0$ (resp. $F$) the unique maximal flat connecting
$[W_p](\infty)$ (resp. $[W_a](\infty)$) and $[W_q](\infty)$ (resp.
$[W_b](\infty)$). Then, it is obvious that $F_n =\gamma_n F_0$.

Since $\gamma_n [W_p] (\infty)$ and $\gamma_n [W_q] (\infty)$
converge to $[W_a](\infty)$ and $[W_b](\infty)$ respectively,
$F_n$ should converge to $F$. More precisely, there is a sequence
$(o_n) \in F_n$ and $o \in F$ such that $(o_n,F_n)$ converges to
$(o,F)$ in the space of pointed flats. This implies that for any
$C>0$ there exists $N>0$ such that $$B(o,R)\cap F_n \neq
\emptyset$$ for all $n \geq N$. Thus we have $d(o,F_n)=d(o,\gamma_n
F_0)=d(\gamma_n^{-1}o, F_0) <C$. In other words, the sequence
$(\gamma_n^{-1}o)$ remains at a bounded distance $C$ of the
maximal flat $F_0$. By the discreteness of $\Gamma$,
$(\gamma_n^{-1}o)$ can not accumulate to a point in $X$. Hence it
should converge to a boundary point in $\partial X$ and moreover,
the boundary point is in $F_0(\infty)$ due to
$d(\gamma_n^{-1}o,F_0)<C$ for all sufficiently large $n$.

\begin{figure}
\includegraphics[scale=0.7]{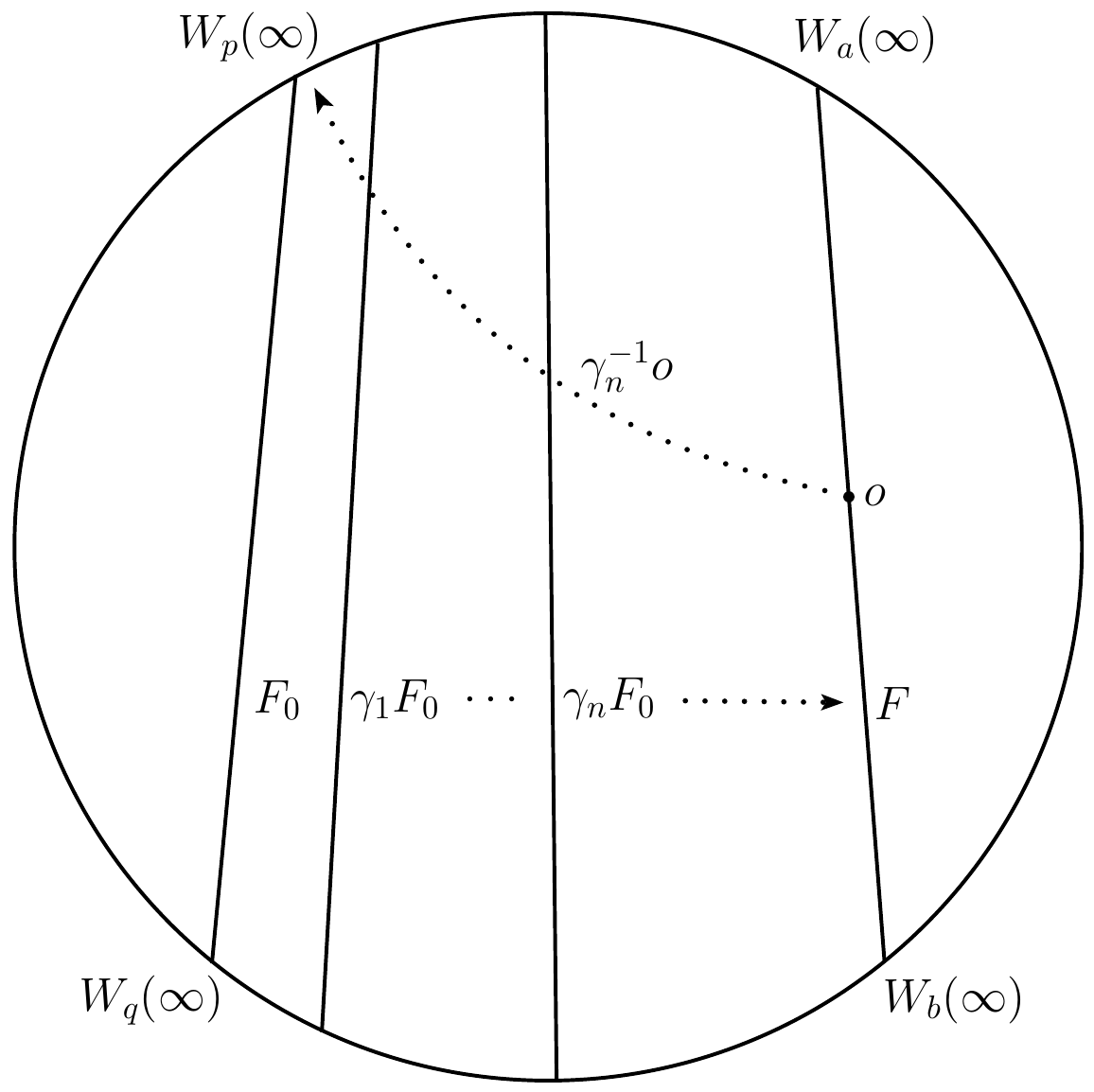}
\caption{Limit set and Weyl chambers.}
\end{figure}

Suppose that $(\gamma_n^{-1}o)$ converges to a point $z \in
F_0(\infty)$. Because $z$ is an accumulation point of the
$\Gamma$-orbit of the point $o\in X$, it should be in the limit
set $L_\Gamma$. According to Theorem \ref{limitset}, $W_p$
and $W_q$ are opposite in $F_0$ and $z$ should be in either
$W_p(\infty)$ or $W_q(\infty)$. Noting that all arguments
above hold for any $q \in \partial \Omega \backslash \{p\}$, one can easily see that $z$ should be in
$W_p(\infty)$. See the Figure $1$.

Let $Pr : X \rightarrow F_0$ be the orthogonal projection onto
$F_0$. Since the sequence $(\gamma_n^{-1}o)$
remains at a bounded distance $C$ of the maximal flat $F_0$, we have
$$d(\gamma_n^{-1}o, Pr(\gamma_n^{-1}o)) < C.$$
This implies that the sequence $(Pr(\gamma_n^{-1}o))$ also
converges to $z$. Since every limit point in $[W_p](\infty)$ is
regular as we mentioned before,
the sequence $(Pr(\gamma_n^{-1}o))$ lies in $W_p^0$ for all large
$n$ where $W_p^0$ is a Weyl chamber in $F_0$ representing $[W_p]$. This implies $$d(\gamma_n^{-1}o,
W_p^0)=d(\gamma_n^{-1}o,Pr(\gamma_n^{-1}o))<C.$$ Finally, we
can conclude that $z$ is a radial limit point in $W_p(\infty)$. This completes the proof.
\end{proof}

\begin{thm} Let $\Gamma$ be a discrete subgroup of $\mathrm{SL}(3,\mathbb R)$ arising from a convex projective structure on a closed surface. Then its limit set in the
geometric boundary of $\mathrm{SL}(3,\br)/\mathrm{SO}(3)$ consists of only
horospherical limit points.
\end{thm}

\begin{proof}
We stick to the notation in the proof of Proposition \ref{prop:6.1}.
Let $w$ be a limit point in $W_p(\infty)$ and $z$ a radial limit point in $W_p(\infty)$. Let
$\mathcal H$ be a horoball based at $w$. The Tits distance
$Td(z,w)$ is less than $\pi/3$. Let $\sigma_1$ be the geodesic ray
emanating from a point of $F_0$ and tending to $z$. Let $\sigma_2$
be a geodesic such that $$\sigma_2(\infty)=w, \ \mathcal H=
b(\sigma_2)^{-1}((-\infty,0)),$$ where $b(\sigma_2) : X
\rightarrow \br$ is the Busemann function associated with
$\sigma_2$. To prove the theorem, we start by observing the following lemma.

\begin{lemma}\label{SequenceInFlat}
Let $F$ be a maximal flat in $X$ and $(x_n)$ be a sequence of
points in $F$ that converges to $z$ in $F(\infty)$. Let $\sigma :
[0,\infty) \rightarrow X$ be a geodesic ray in $F$ tending to $z$.
Then, for any $\epsilon>0$, there exist a sequence $(t_n)$ in
$[0,\infty)$ and $N>0$ such that $$d(\sigma(t_n), x_n) < \epsilon
t_n,$$ for all $n\geq N$.
\end{lemma}
\begin{proof}
Let $\sigma(t_n)$ be the projection point of
$x_n$ onto the geodesic ray $\sigma$. Denote by $\angle_{\sigma(0)}(x_n,
\sigma(t_n))$ the angle subtended at $\sigma(0)$ by $x_n$ and
$\sigma(t_n)$. Then, $\angle_{\sigma(0)}(x_n, \sigma(t_n))$ converges to
zero by the definition of the cone topology on $X\cup \partial X$.
Hence we have $$ \tan \angle_{\sigma(0)}(x_n, \sigma(t_n)) =
\frac{d(\sigma(t_n),x_n)}{t_n} \rightarrow 0 \text{ as }
n\rightarrow \infty .$$ Thus, for a given $\epsilon >0$, there
exists $N>0$ such that
$$ \frac{d(\sigma(t_n),x_n)}{t_n} < \epsilon.$$
for all $n \geq N$. Therefore $d(\sigma(t_n), x_n) < \epsilon t_n$ for all $n\geq N$.
\end{proof}

According to Lemma \ref{SequenceInFlat}, for any $\epsilon>0$,
there exist a sequence $(t_n)$ in $[0,\infty)$ and $N>0$ such that
$$d(\sigma_1(t_n), Pr(\gamma_n^{-1}o)) < \epsilon t_n,$$ for all
$n\geq N$. Choose an $\epsilon < \cos (Td(z,w))$. Since the norms
of gradient vectors of the Busemann function $b(\sigma_2)$ are
equal to $1$, we have $$\left| b(\sigma_2)(\gamma_n^{-1}o)-
b(\sigma_2)(Pr(\gamma_n^{-1}o)) \right| \leq d(\gamma_n^{-1}o,
Pr(\gamma_n^{-1}o)) <C.$$ In the same way, we obtain
$$\left| b(\sigma_2)(\sigma_1(t_n)) - b(\sigma_2)(Pr(\gamma_n^{-1}o)) \right|
\leq d(\sigma_1(t_n),Pr(\gamma_n^{-1}o)) <\epsilon t_n,$$
for all $n \geq N$. Furthermore, since $Td(z,w)<\pi/3$, it follows from \cite[Lemma 3.4]{Ha} that
$$b(\sigma_2)(\sigma_1(t)) < -t \cdot \cos(Td(z,w))+D,$$
for some $D>0$. Then for all $n \geq N$,
{\setlength\arraycolsep{2pt}
\begin{eqnarray*}
b(\sigma_2) (\gamma_n^{-1}o) &=& b(\sigma_2)(\gamma_n^{-1}o)-b(\sigma_2)(Pr(\gamma_n^{-1}o)) \\
&+& b(\sigma_2)(Pr(\gamma_n^{-1}o)) - b(\sigma_2)(\sigma_1(t_n)) +b(\sigma_2)(\sigma_1(t_n)) \\
&<& C + \epsilon t_n - t_n \cdot \cos(Td(z,w))+D \\
&<& C -t_n(1/2 - \epsilon) +D.
\end{eqnarray*}}
Since the sequence $(t_n)$ goes to infinity, one can choose a sufficiently large
$N>0$ such that for all $n \geq N$,
$$b(\sigma_2) (\gamma_n^{-1}o)<0.$$
Hence $\gamma_n^{-1}o \in \mathcal H$ for all $n \geq N$.
Therefore, we can conclude that $w$ is a horospherical limit
point.
\end{proof}

\section{Further characterization of limit points in convex real projective surfaces}

In the previous section, we prove that every Weyl chamber at
infinity with nonempty limit set of a discrete group $\Gamma$ in
the Hitchin component for $\mathrm{SL}(3,\mathbb R)$ has at least
one radial limit point. One can ask how many limit points in each
Weyl chamber at infinity are radial limit points. In this section,
we answer this question and describe where the set of radial limit
points is positioned in the limit set of $\Gamma$.

\begin{thm}\label{thm:7.1}
Let $\Gamma$ be a discrete subgroup of $\mathrm{SL}(3,\mathbb R)$
arising from a convex projective structure on a closed surface.
Then there is only one radial limit point in any Weyl chamber at
infinity with nonempty limit set of $\Gamma$ in the geometric
boundary of $\mathrm{SL}(3,\br)/\mathrm{SO}(3)$.
\end{thm}

\begin{proof}
Let $W$ be a Weyl chamber in $X=\mathrm{SL}(3,\mathbb
R)/\mathrm{SO}(3)$ such that the limit set of $\Gamma$ in
$W(\infty)$ is nonempty. As we observed before, we can assume
$W=\phi(p)=W_p$ for some $p \in \partial \Omega$. Suppose that $z$
is a radial limit point in $W_p(\infty)$. Fix a point $o\in X$. By
the definition of radial limit point, there exists a sequence
$(\gamma_n)_{n\in \mathbb N}$ of $\Gamma$  and a constant $C>0$
such that \begin{eqnarray}\label{eqn:1} d(\gamma_n o, W_p) <
C\end{eqnarray} for all $n \in\mathbb N$.

Choose a point $q \in \partial \Omega$ distinct from $p$. Since
$W_p(\infty)$ and $W_q(\infty)$ are opposite, there
exists a unique maximal flat $F_0$ connecting $W_p(\infty)$ and
$W_q(\infty)$. Inequality (\ref{eqn:1}) implies that for some
$C_0>0$,
$$d(\gamma_n o, F_0) < C_0.$$
In other words, $d(o, \gamma_n^{-1} F_0) < C_0.$ Since the
sequence $(\gamma_n^{-1} F_0)$ of maximal flats remains at a
bounded distance of a point $o \in X$, it converges to a maximal flat
$F$ in the space of pointed maximal flats (See
\cite[Section 8.3 and 8.4]{Eb}). Then the sequences
$\gamma_n^{-1} W_p(\infty)$ and $\gamma_n^{-1} W_q(\infty)$
converge to $W_a(\infty)$ and $W_b(\infty)$ in
$F(\infty)$ respectively for some $a,b\in \partial \Omega$.
Furthermore, since $\gamma_n^{-1}
W_p(\infty)$ and $\gamma_n^{-1} W_q(\infty)$ are opposite, $W_a(\infty)$ and  $W_b(\infty)$
should be opposite. This implies $ p \neq q$ and thus, there exists a unique geodesic joining $a$
and $b$ in $\Omega$.

On the other hand, due to the $\Gamma$-equivariant map $\phi :
\partial \Omega  \rightarrow \Lambda_\Gamma$, the sequences
$(\gamma_n^{-1}p)$ and $(\gamma_n^{-1}q)$ converge to $a$
and $b$ respectively. Let $l_{pq}$ be the
geodesic connecting $p$ and $q$ with respect to the Hilbert metric on $\Omega$. Then the
sequence $(\gamma_n^{-1} l_{pq})$ converges to $l_{ab}$ and thus,
for a point $e \in l_{ab}$, there is a constant $D>0$ such
that for all large $n$, $$d(\gamma_n e, l_{pq})=d(e,\gamma_n^{-1} l_{pq})<D.$$ Hence, the
sequence $(\gamma_n e)$ should converge to either $p$ or $q$.
Noting that all arguments above hold for any $q \in \Omega \backslash \{p\}$, it can be easily seen that $(\gamma_n e)$
converges to $p$. Furthermore, since the sequence $(\gamma_n e)$ remains at a bounded distance of $l_{pq}$, the
broken geodesic ray consisting of geodesic segments $[\gamma_n
e, \gamma_{n+1} e]$ becomes a quasi-geodesic ray in $\Omega$ by choosing its subsequence so that the distance
between any two points is greater than a sufficiently large
constant.

Now, let's consider the broken geodesic ray $\mathcal R$ in the Cayley
graph $C_\Gamma$ consisting of geodesic segments $[\gamma_n,
\gamma_{n+1}]$. Then the ray $\mathcal R$ is also a quasi-geodesic ray
in the Cayley graph $C_\Gamma$ because $C_\Gamma$ and $\Omega$ are
quasi-isometric by an orbit map. Via a canonical identification
between $\partial C_\Gamma$ and $\partial \Omega$, we can assume
that the sequence $\gamma_n$ converges to $p \in \partial
C_\Gamma$. According to the Morse lemma, the quasi-geodesic ray
$\mathcal R$ remains at a bounded Hausdorff distance of a geodesic ray
in $ C_\Gamma$ joining $id$ and $p$ where $id$ is the
identity element of $\Gamma$.

Suppose that $z'$ is another radial limit point in $W_p(\infty)$
and $\gamma_n'o$ converges to $z'$ with $d(\gamma_n' o, W_p)<C'$
for some $C'>0$. In the same way as above, we get a quasi-geodesic ray
$\mathcal R'$ in $C_\Gamma$ consisting of geodesic segments
$[\gamma_n', \gamma_{n+1}']$ whose endpoint is $p \in \partial C_\Gamma$.
Noting that $\mathcal R$ and $\mathcal R'$ are quasi-geodesic rays in $C_\Gamma$ with the same endpoint,
it can be easily seen by the Morse lemma that $\mathcal R'$ remains at a bounded Hausdorff distance of $\mathcal R$.
Moreover, since the orbit map $C_\Gamma \rightarrow
X$ is a quasi-isometric embedding (see section \ref{well}), two quasi-geodesic rays $\mathcal R o$ and $\mathcal R' o$ in $X$
should have the same endpoint at infinity.
This means that $z=z'$. Therefore, $W_p(\infty)$ contains exactly one
radial limit point.
\end{proof}

Theorem \ref{limitset} and \ref{thm:7.1} imply that the set of
radial limit points of $\Gamma$ is isomorphic to a circle $S^1$ in
the category of sets. Moreover, Link \cite{Ga06} proved that if
$\Gamma$ is a non-elementary discrete group, then the set of
attracting fixed points of regular axial isometries is a dense
subset of the limit set $L_\Gamma$. Since the set of radial limit
points contains the set of attractive fixed points of regular
axial isometries, the set of radial limit points is also dense in
$L_\Gamma$. Hence we have the following immediate corollary.

\begin{co}
Let $\Gamma$ be a discrete subgroup of $\mathrm{SL}(3,\mathbb R)$
arising from a convex projective structure on a closed surface.
Then the set of radial limit points of $\Gamma$ in the geometric
boundary of $\mathrm{SL}(3,\br)/\mathrm{SO}(3)$ is isomorphic to a
circle $S^1$ and dense in the limit set $L_\Gamma$ of $\Gamma$.
\end{co}

\section{Anosov representations in semisimple Lie group}\label{Anosov}
A Fuchsian representation from $\pi_1(S)$ to $\mathrm{PSL}(n,\br)$, where
$S$ is a closed surface with genus $\geq 2$, is a representation
$\rho=\iota \circ \rho_0$, where $\rho_0$ is a Fuchsian
representation in $\mathrm{PSL}(2,\br)$ and $\iota$ is the irreducible
representation of $\mathrm{PSL}(2,\br)$ in $\mathrm{PSL}(n,\br)$. A Hitchin
component is the connected component of a representation variety
which contains fuchsian representations.  In \cite{La}, it is
shown that a Hitchin representation is hyperconvex and vice versa
in the following sense; a representation $\rho:\pi_1(S)\ra
\mathrm{PSL}(n,\br)$ is hyperconvex if there exists a $\rho$-equivariant
hyperconvex curve $\xi$ from $\partial_\infty\pi_1(S)$ in
$\RP^{n-1}$, i.e., for any distinct points $(x_1,\ldots,x_n)$, the
sum $\xi(x_1)+\cdots + \xi(x_n)$ is direct. Such $\xi$ is unique
and $\xi$ is called the limit curve of $\rho$. In \cite{La}, it is
shown that a Hitchin representation $\rho$ is hyperconvex and
discrete, faithful. Furthermore, $\rho(\gamma), id\neq \gamma\in
\pi_1(S)$ is real split with distinct eigenvalues. If $\gamma^+$
is the attracting fixed point of $\gamma$ in
$\partial_\infty\pi_1(S)$, then $\xi(\gamma^+)$ is the unique
attracting fixed point of $\rho(\gamma)$ in $\RP^{n-1}$.
Furthermore the limit curve $\xi$ is a hyperconvex Frenet curve:
there exists a family $(\xi=\xi^1,\xi^2,\ldots,\xi^{n-1})$ called
the osculating flag so that
\begin{enumerate}
\item{$\xi^p$ takes values in the Grassmannian of $p$-planes,}
\item{$\xi^p(x)\subset \xi^{p+1}(x)$,} \item{if $(n_1,\ldots,n_l)$
are positive integers such that $\sum n_i\leq n$ and if
$(x_1,\ldots,x_l)$ are distinct points, the sum
$$\xi^{n_1}(x_1)+\cdots + \xi^{n_l}(x_l)$$ is direct.}
\item{If $p=n_1+\cdots + n_l$, then for all distinct points $(y_1,\ldots,y_l)$,
$$\lim_{(y_1,\ldots,y_l)\ra x}\oplus \xi^{n_i}(y_i)=\xi^p(x).$$}
\end{enumerate}

Specially if we take $x\neq y$ on $\partial \pi_1(S)$, then for any $n_1+n_2=n$, $\xi^{n_1}(x)+\xi^{n_2}(y)=\br^n$.
Hence for such a representation in a Hitchin component, any two distinct points on the ideal boundary define two opposite Weyl chambers
in $\mathrm{ SL}(n,\br)/\mathrm{SO}(n)$. Then all the previous arguments work in this case also. Hence Theorem \ref{thm:1.5} immediately follows.

In \cite{GW}, this notion is generalized to the finitely generated word hyperbolic group $\Gamma$ as follows. A representation
$\rho:\Gamma\ra G$ to a semisimple Lie group $G$ is $P^+$-Anosov if there exist continuous $\rho$-equivariant maps
$\psi^+:\partial \Gamma\ra G/P^+$ and $\psi^-:\partial \Gamma\ra G/P^-$, where $P^\pm$  are two opposite parabolic subgroups, such that
\begin{enumerate}
\item for all $(t,t')\in \partial \Gamma\times \partial \Gamma \setminus \triangle$, the pair $(\psi^+(t),\psi^-(t'))$ is in the unique open $G$-orbit in $G/P^+\times G/P^-$. Here $\triangle$ is a diagonal set.
\item for all $t\in \partial \Gamma$, the pair $(\psi^+(t),\psi^-(t))$ is contained in a unique closed $G$-orbit in $G/P^+\times G/P^-$.
\item they satisfy some contraction property with respect to the flow.
\end{enumerate}

Two such examples are;
\begin{enumerate}
\item Let $G$ be a split real simple Lie group and $S$ be a closed oriented surface of genus $\geq 2$. Representations $\rho:\pi_1(S)\ra G$ in the Hitchin component are $B$-Anosov where $B\subset G$ is a Borel subgroup.
\item The holonomy representation $\rho:\pi_1(M)\ra \mathrm{PGL}(n+1,\br)$ of a strictly convex real projective structure on an $n$-dimensional manifold $M$ is $P$-Anosov where $P\subset \mathrm{PGL}(n+1,\br)$ is the stabilizer of a line.
\end{enumerate}

Hence if $\rho:\Gamma\ra G$ is a $P$-Anosov representation from a word hyperbolic group $\Gamma$ where $P$ is a minimal parabolic subgroup of a semisimple Lie group $G$. In this case $P=P^+=MAN$ and its opposite minimal parabolic subgroup $P^-=MAN^-$ are conjugate. Hence two spaces $G/P$ and $G/P^-$ are canonically identified with the set of Weyl chambers at infinity. Then by
the first property of the maps $\psi^\pm$, for any distinct elements $t,t'\in \partial \Gamma$, $\psi^+(t)$ and $\psi^-(t')$ are
opposite Weyl chambers of the symmetric space $G/K$.
Specially by the uniqueness of $\psi^\pm$, $\psi^+=\psi^-$ (see section 4.5 in \cite{GW}), and hence we obtain
\begin{lemma}The map $\psi^+$ is injective and for any distinct elements $t,t'\in \partial \Gamma$, $\psi^+(t)$ and $\psi^+(t')$ are
opposite Weyl chambers of the symmetric space $G/K$. Furthermore the orbit map $\Gamma\ra G/K$ is a quasi-isometric embedding.
\end{lemma}
\begin{proof}If $\psi^+(t)=\psi^+(t')$ for two distinct elements $t$ and $t'$, since $\psi^+=\psi^-$ it will contradicts the fact that $\psi^+(t)$ and $\psi^-(t')$ are opposite.
The second statement follows from the property (ii) of Theorem 1.7 in \cite{GW}.
\end{proof}

Therefore all the previous arguments work in this more general context as well.
\begin{thm}
Let
$\rho:\Gamma\ra G$ be a Zariski dense discrete $P$-Anosov representation from a word hyperbolic group $\Gamma$ where $P$ is a minimal parabolic subgroup of a semisimple Lie group $G$.  Then the geometric limit set is isomorphic to the set $\partial \Gamma\times \partial \cal L_{\rho(\Gamma)}$. Furthermore in each Weyl chamber intersecting the geometric limit set nontrivially, there is only one radial limit point.
\end{thm}

\end{document}